\documentclass[12pt]{article}
\usepackage{amsthm,amssymb,amsmath}
\usepackage{epsfig}
\usepackage{color}
\usepackage[geometry]{ifsym}
\usepackage{amsmath} 
\usepackage{verbatim}
\usepackage{fancyhdr}
\usepackage{graphicx}
\usepackage{mathrsfs}
\usepackage{pst-poly}
\usepackage{cite}   

\newtheorem{prelem}{{\bf Proposition}}

\newenvironment{lem}{\begin{prelem}{\hspace{-0.5em}{\bf}}}{\end{prelem}}
 \newtheorem{theorem}{Theorem}
\newtheorem{corollary}[theorem]{Corollary}
\newtheorem{lemma}[theorem]{Lemma}
\newtheorem{observation}[theorem]{Observation}

\theoremstyle{definition}

\theoremstyle{remark}

\title{On the Roman bondage number of a graph\thanks {The work was supported by NNSF
of China (No. 11071233).}}
\author {A. Bahremandpour$^a$, Fu-Tao Hu$^b$, S.M. Sheikholeslami$^a$, Jun-Ming Xu$^c$\thanks {Corresponding author:
xujm@ustc.edu.cn (J.-M. Xu)}\vspace{2mm}\\ \\
$^a$Department of Mathematics \\
Azarbaijan University of Tarbiat Moallem\\
Tabriz, I.R. Iran\\
{\tt s.m.sheikholeslami@azaruniv.edu} \\ \\
$^b$School of Mathematical Sciences\\
Anhui University\\
Hefei, Anhui, 230601, China\\
{\tt hufu@mail.ustc.edu.cn}\\ \\
$^c$School of Mathematical Sciences\\
University of Science and Technology of China\\
Wentsun Wu Key Laboratory of CAS\\
Hefei, Anhui, 230026, China
}
\date{}
\begin{document}
\maketitle

\begin{abstract}
A {\em Roman dominating function} on a graph $G=(V,E)$ is a
function $f:V\rightarrow\{0,1,2\}$ such that every vertex $v\in V$
with $f(v)=0$ has at least one neighbor $u\in V$ with $f(u)=2$.
The {\em weight} of a Roman dominating function is the value
$f(V(G))=\sum_{u\in V(G)}f(u)$. The minimum weight of a Roman
dominating function on a graph $G$ is called the {\em Roman
domination number}, denoted by $\gamma_{R}(G)$.  The Roman bondage
number $b_{R}(G)$ of a graph $G$ with maximum degree at least two
is the minimum cardinality of all sets $E'\subseteq E(G)$ for
which $\gamma_{R}(G-E')>\gamma_R(G)$. In this paper, we first show
that the decision problem for determining $b_{\rm R}(G)$ is
NP-hard even for bipartite graphs and then we establish some sharp
bounds for $b_{\rm R}(G)$ and characterizes all graphs attaining
some of these bounds. \vspace{4mm}

\end{abstract}
\section{Introduction}
For terminology and notation on graph theory not given here, the
reader is referred to \cite{hh197,hh297,x03}.  In this paper, $G$
is a simple graph with vertex set $V=V(G)$ and edge set $E=E(G)$.
The order $|V|$ of $G$ is denoted by $n=n(G)$. For every vertex
$v\in V$, the {\em open neighborhood} $N(v)$ is the set $\{u\in
V\mid uv\in E\}$ and the {\em closed neighborhood} of $v$ is the
set $N[v] = N(v) \cup \{v\}$. The {\em degree} of a vertex $v\in
V$ is $\deg_G(v)=\deg(v)=|N(v)|$. The {\em minimum} and {\em
maximum degree} of a graph $G$ are denoted by $\delta=\delta(G)$
and $\Delta=\Delta(G)$, respectively. The {\em open neighborhood}
of a set $S\subseteq V$ is the set $N(S)=\cup_{v\in S}N(v)$, and
the {\em closed neighborhood} of $S$ is the set $N[S]=N(S)\cup S$.
The {\em complement} $\overline{G}$ of $G$ is the simple graph
whose vertex set is $V$ and whose edges are the pairs of
nonadjacent vertices of $G$.  We write $K_n$ for the {\em complete
graph} of order $n$ and $C_n$ for a {\em cycle} of length $n$. For
two disjoint nonempty sets $S,T\subset V(G)$, $E_G(S,T)=E(S,T)$
denotes the set of edges between $S$ and $T$.

A subset $S$ of vertices of $G$ is a {\em dominating set~} if
$|N(v)\cap S|\ge 1$ for every $v\in V-S$. The {\em domination
number} $\gamma(G)$  is the minimum cardinality of a dominating
set of $G$.  To measure the vulnerability or the stability of the
domination in an interconnection network under edge failure, Fink
et at. \cite{fjkr90} proposed the concept of the bondage number in
1990. The {\it bondage number}, denoted by $b(G)$, of $G$ is the
minimum number of edges whose removal from $G$ results in a graph
with larger domination number. An edge set $B$ for which
$\gamma(G-B)>\gamma(G)$ is called a {\it bondage set}. A
$b(G)$-set is a  bondage set of $G$ of size $b(G)$. If $B$ is a
$b(G)$-set, then obviously
\begin{equation}\label{eqb}\gamma(G-B)=\gamma(G)+1.\end{equation}

A {\em Roman dominating function} on a graph $G$ is a labeling
$f:V\rightarrow \{0, 1, 2\}$ such that every vertex with label 0
has at least one  neighbor with label 2. The weight of a Roman
dominating function is the value $f(V(G))=\sum_{u\in V(G)}f(u)$,
denoted by $f(G)$. The minimum weight of a Roman dominating
function on a graph $G$ is called the {\em Roman domination
number}, denoted by $\gamma_{R}(G)$.  A $\gamma_{R}(G)$-{\em
function} is a Roman dominating function on $G$ with weight
$\gamma_{R}(G)$. A Roman dominating function $f : V\rightarrow
\{0, 1, 2\}$ can be represented by the ordered partition $(V_0,
V_1, V_2)$ (or $(V_{0}^{f},V_{1}^{f},V_{2}^{f})$ to refer to $f$)
of $V$, where $V_i=\{v\in V\mid f(v) = i\}$. In this
representation, its weight is $\omega(f)=|V_1|+2|V_2|$. It is
clear that $V_1^f\cup V_2^f$ is a dominating set of $G$, called
{\it the Roman dominating set}, denoted by $D^f_{\rm
R}=(V_1,V_2)$. Since $V_1^f\cup V^f_2$ is a dominating set when
$f$ is an RDF, and since placing weight 2 at the vertices of a
dominating set yields an RDF, in \cite{cd04}, it was observed that
\begin{equation}\label{eqq}\gamma(G)\le \gamma_{R}(G)\le 2\gamma(G).\end{equation}
A graph $G$ is called to be {\it Roman} if $\gamma_{\rm
R}(G)=2\gamma(G)$.

The definition of the Roman dominating function was given
implicitly by Stewart \cite{s} and ReVelle and Rosing \cite{rr}.
Cockayne, Dreyer Jr., Hedetniemi and Hedetniemi \cite{cd04} as
well as  Chambers, Kinnersley, Prince and West \cite{ck10} have
given a lot of results on Roman domination. For more information
on Roman domination we refer the reader to \cite{ck10, cd04,CFM,
FKS, fy09, H,H1,HS, lk08, lkl05, pp02, rs07, sh07, sh10, xc06}.

Let  $G$ be a graph with maximum degree at least two. The {\em
Roman bondage number} $b_{R}(G)$ of $G$ is the minimum cardinality
of all sets $E'\subseteq E$ for which
$\gamma_{R}(G-E')>\gamma_{R}(G)$. Since in the study of Roman
bondage number the assumption $\Delta(G)\ge 2$ is necessary, we
always assume that when we discuss $b_R(G)$, all graphs involved
satisfy $\Delta(G)\ge 2$. The Roman bondage number $b_R(G)$ was
introduced by Jafari Rad and Volkmann in \cite{RV1}, and has been
further studied for example in \cite{AQ, DKSV, DSV, EP, HX, RV2}.

An edge set $B$ that $\gamma_{\rm R}(G-B)>\gamma_{\rm R}(G)$ is
called the {\it Roman bondage set}. A $b_R(G)$-set is a Roman
bondage set of $G$ of size $b_R(G)$. If $B$ is a $b_R(G)$-set,
then clearly
\begin{equation}\label{eqrb}\gamma_{\rm R}(G-B)=\gamma_{\rm R}(G)+1.\end{equation}

In this paper, we first show that the decision problem for
determining $b_{\rm R}(G)$ is NP-hard even for bipartite graphs
and then we establish some sharp bounds for $b_{\rm R}(G)$ and
characterizes all graphs attaining some of these bounds.

We make use of the following results in this paper.

\begin{prelem}\label{propA}
{\em (Chambers et al. \cite{ck10})} If G is a graph of order $n$,
then $\gamma_R(G)\le n-\Delta(G) + 1$.
\end{prelem}

\begin{lem}\label{propB}
{\rm (Cockayne et al. \cite{cd04})} For a grid graph $P_2\times
P_n$,
 $$
 \gamma_{\rm R}(P_2\times P_n)=n+1.
 $$
\end{lem}

\begin{lem}\label{propC}
{\rm (Cockayne et al. \cite{cd04})}\ For any graph $G$,
$\gamma(G)\leq \gamma_{\rm R}(G)\leq 2\gamma(G)$.
\end{lem}

\begin{lem}\label{propD}
{\rm (Cockayne et al. \cite{cd04})}\ For any graph $G$ of order n,
$\gamma(G)=\gamma_{\rm R}(G)$ if and only if $G=\bar{K_n}$.
\end{lem}

\begin{lem}\label{propE}
{\rm (Cockayne et al. \cite{cd04})}\ If $G$ is a connected graph
of order n, then $\gamma_{\rm R}(G) =\gamma(G)+1$ if and only if
there is a vertex $v\in V(G)$ of degree $n-\gamma(G)$.
\end{lem}

\begin{lem}\label{propF}
{\rm (Hu and Xu \cite{propF})} If $G=K_{3,3,\ldots,3}$ is the
complete $t$-partite graph of order $n\ge 9$, then $b_{\rm
R}(G)=n-1$.
\end{lem}

\begin{lem}\label{propG}
{\rm (Jafari Rad and Volkmann \cite{RV1})} If $G$ is a connected
graph of order $n\ge 3$, then $b_{\rm R}(G)\le
\delta(G)+2\Delta(G)-3$.
\end{lem}


\begin{lem}\label{propH}
{\em {\rm (Fink et al.~\cite{fjkr90}, Rad and
Volkmann~\cite{RV1})} For a cycle $C_n$ of order $n$,
$$
b(C_n)= \left\{ \begin{array}{ll}
3, & {\rm if}\ n=1\, ({\rm mod}\, 3);\\
2, & {\rm otherwise}.
 \end{array} \right.
 $$

$$
b_{\rm R}(C_n)= \left\{ \begin{array}{ll}
3, & {\rm if}\ n=2\, ({\rm mod}\, 3);\\
2, & {\rm otherwise}.
 \end{array} \right.
 $$}
\end{lem}

\begin{observation}\label{ob1}
Let $G$ be a connected graph of order $n\ge 3$. Then $\gamma_{\rm
R}(G)=2$ if and only if $\Delta(G)=n-1$.
\end{observation}


\begin{observation}\label{ob2}
Let $G$ be a graph of order $n$ with maximum degree at least two.
Assume that $H$ is a spanning subgraph of $G$ with $\gamma_{\rm
R}(H)=\gamma_{\rm R}(G)$. If $K=E(G)-E(H)$, then $b_{\rm R}(H)\le
b_{\rm R}(G)\le b_{\rm R}(H)+|K|$.
\end{observation}

\begin{lem}\label{propG}
Let $G$ be a nonempty graph of order $n\geq 3$, then $\gamma_{\rm
R}(G)=3$ if and only if $\Delta(G)=n-2$.
\end{lem}

\begin{proof}
Let $\Delta(G)=n-2$. Assume that $u$ is a vertex of degree $n-2$ and
$v$ is the unique vertex not adjacent to $u$ in $G$. By Observation
\ref{ob1},  $\gamma_{\rm R}(G)\geq 3$ and clearly
$f=(V(G)-\{u,v\},\{v\},\{u\})$ is a Roman dominating set of $G$ with
$f(G)=3$. Thus, $\gamma_{\rm R}(G)=3$.

Conversely, assume $\gamma_{\rm R}(G)=3$. Then $\Delta(G)\leq n-2$
by Proposition \ref{propA}. Let $f=(V_0,V_1,V_2)$ be a $\gamma_{\rm
R}$-function of $G$. If $V_2=\emptyset$, then $f(v)=1$ for each
vertex $v\in V(G)$, and hence $n=3$. Sine $G$ is nonempty and
$\Delta(G)\leq n-2=1$, we have $\Delta(G)=n-2=1$. Let
$V_2\neq\emptyset$. Since $\gamma_{\rm R}(G)=3$, we deduce that
$|V_1|=|V_2|=1$. Suppose $V_1=\{v\}$ and $V_2=\{u\}$. Then other
$n-2$ vertices assigned $0$ are must be  adjacent to $u$. Thus,
$\Delta(G)\geq d_G(u)\geq n-2$ and hence $\Delta(G)=n-2$.
\end{proof}

\section{Complexity of Roman bondage number}

In this section, we will show that the Roman bondage number
problem is NP-hard and the Roman domination number problem is
NP-complete even for bipartite graphs. We first state the problem
as the following decision problem.

\begin{center}
\begin{minipage}{130mm}
\setlength{\baselineskip}{24pt}

\vskip6pt\noindent {\bf Roman bondage number problem (RBN):}

\noindent {\bf Instance:}\ {\it A nonempty bipartite graph $G$ and
a positive integer $k$.}

\noindent {\bf Question:}\ {\it Is $b_{\rm R}(G)\le k$?}

\end{minipage}
\end{center}

\begin{center}
\begin{minipage}{130mm}
\setlength{\baselineskip}{24pt}

\vskip6pt\noindent {\bf Roman domination number problem (RDN):}

\noindent {\bf Instance:}\ {\it A nonempty bipartite graph $G$ and
a positive integer $k$.}

\noindent {\bf Question:}\ {\it Is $\gamma_{\rm R}(G)\le k$?}

\end{minipage}
\end{center}

Following Garey and Johnson's techniques for proving
NP-completeness given in~\cite{gj79}, we prove our results by
describing a polynomial transformation from the known-well
NP-complete problem: 3SAT. To state 3SAT, we recall some terms.

Let $U$ be a set of Boolean variables. A {\it truth assignment}
for $U$ is a mapping $t: U\to\{T,F\}$. If $t(u)=T$, then $u$ is
said to be ``\,true" under $t$; If $t(u)=F$, then $u$ is said to
be ``\,false" under $t$. If $u$ is a variable in $U$, then $u$ and
$\bar{u}$ are {\it literals} over $U$. The literal $u$ is true
under $t$ if and only if the variable $u$ is true under $t$; the
literal $\bar{u}$ is true if and only if the variable $u$ is
false.

A {\it clause} over $U$ is a set of literals over $U$. It
represents the disjunction of these literals and is {\it
satisfied} by a truth assignment if and only if at least one of
its members is true under that assignment. A collection $\mathscr
C$ of clauses over $U$ is {\it satisfiable} if and only if there
exists some truth assignment for $U$ that simultaneously satisfies
all the clauses in $\mathscr C$. Such a truth assignment is called
a {\it satisfying truth assignment} for $\mathscr C$. The 3SAT is
specified as follows.

\begin{center}
\begin{minipage}{130mm}
\setlength{\baselineskip}{24pt}

\vskip6pt\noindent {\bf $3$-satisfiability problem (3SAT):}

\noindent {\bf Instance:}\ {\it A collection
$\mathscr{C}=\{C_1,C_2,\ldots,C_m\}$ of clauses over a finite set
$U$ of variables such that $|C_j| =3$ for $j=1, 2,\ldots,m$.}

\noindent {\bf Question:}\ {\it Is there a truth assignment for
$U$ that satisfies all the clauses in $\mathscr{C}$?}

\end{minipage}
\end{center}

\begin{theorem}\label{thm3}
{\em (Theorem 3.1 in~\cite{gj79})} 3SAT is NP-complete.
\end{theorem}

\begin{theorem}\label{thm4}
RBN is NP-hard even for bipartite graphs.
\end{theorem}

\begin{proof}
The transformation is from 3SAT. Let $U=\{u_1,u_2,\ldots,u_n\}$
and $\mathscr{C}=\{C_1,C_2,\ldots,$ $C_m\}$ be an arbitrary
instance of 3SAT. We will construct a bipartite graph $G$ and
choose an integer $k$ such that $\mathscr{C}$ is satisfiable if
and only if $b_{\rm R}(G)\leq k$. We construct such a graph $G$ as
follows.

For each $i=1,2,\ldots,n$, corresponding to the variable $u_i\in
U$, associate a graph $H_i$ with vertex set
$V(H_i)=\{u_i,\bar{u}_i,v_i,v_i',x_i,y_i,z_i,w_i\}$ and edge set
$E(H_i)=\{u_iv_i,u_iz_i,\bar{u}_iv_i',\\
\bar{u}_iz_i,y_iv_i,y_iv_i',y_iz_i,w_iv_i,w_iv_i',w_iz_i,
x_iv_i,x_iv_i'\}$. For each $j=1,2,\ldots,m$, corresponding to the
clause $C_j=\{p_j,q_j,r_j\}\in \mathscr{C}$, associate a single
vertex $c_j$ and add edge set $E_j=\{c_jp_j, c_jq_j,c_jr_j\}$,
$1\le j\le m$. Finally, add a path $P=s_1s_2s_3$, join $s_1$ and
$s_3$ to each vertex $c_j$ with $1\le j\le m$ and set $k=1$.

Figure~\ref{f1} shows an example of the graph obtained when
$U=\{u_1,u_2,u_3,u_4\}$ and $\mathscr{C}=\{C_1,C_2,C_3\}$, where
$C_1=\{u_1,u_2,\bar{u}_3\}, C_2=\{\bar{u}_1,u_2,u_4\},
C_3=\{\bar{u}_2,u_3,u_4\}$.

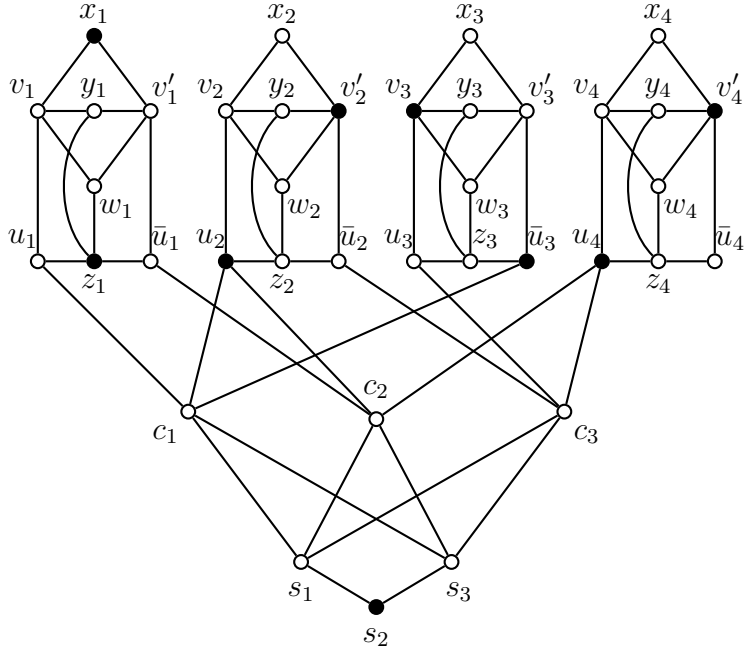
\begin{figure}[h]
\begin{center}
\begin{pspicture}(-5,-1.1)(5,7.7)

\cnode*(0,-.6){3pt}{s2}\rput(0,-1){$s_2$}
\cnode(-1,0){3pt}{s1}\rput(-1,-.4){$s_1$}
\cnode(1,0){3pt}{s3}\rput(1.1,-.4){$s_3$} \ncline{s2}{s1}
\ncline{s2}{s3}

\cnode(0,1.9){3pt}{c2}\rput(0,2.3){$c_2$} \ncline{c2}{s1}
\ncline{c2}{s3} \cnode(-2.5,2){3pt}{c1}\rput(-2.8,1.7){$c_1$}
\ncline{c1}{s1} \ncline{c1}{s3}
\cnode(2.5,2){3pt}{c3}\rput(2.8,1.7){$c_3$} \ncline{c3}{s1}
\ncline{c3}{s3}

\cnode(-4.5,4){3pt}{u1}\rput(-4.7,4.3){$u_1$}
\cnode(-3,4){3pt}{u1'}\rput(-2.8,4.3){$\bar{u}_1$}

\cnode*(-2,4){3pt}{u2}\rput(-2.2,4.3){$u_2$}
\cnode(-0.5,4){3pt}{u2'}\rput(-0.3,4.3){$\bar{u}_2$}

\cnode(0.5,4){3pt}{u3}\rput(0.3,4.3){$u_3$}
\cnode*(2,4){3pt}{u3'}\rput(2.2,4.3){$\bar{u}_3$}

\cnode*(3,4){3pt}{u4}\rput(2.8,4.3){$u_4$}
\cnode(4.5,4){3pt}{u4'}\rput(4.7,4.3){$\bar{u}_4$}

\cnode(-4.5,6){3pt}{v1}\rput(-4.7,6.3){$v_1$}
\cnode(-3,6){3pt}{v1'}\rput(-2.8,6.3){$v_1'$} \ncline{u1}{v1}
\ncline{u1'}{v1'} \cnode(-2,6){3pt}{v2}\rput(-2.2,6.3){$v_2$}
\cnode*(-0.5,6){3pt}{v2'}\rput(-0.3,6.3){$v_2'$} \ncline{u2}{v2}
\ncline{u2'}{v2'} \cnode*(0.5,6){3pt}{v3}\rput(0.3,6.3){$v_3$}
\cnode(2,6){3pt}{v3'}\rput(2.2,6.3){$v_3'$} \ncline{u3}{v3}
\ncline{u3'}{v3'} \cnode(3,6){3pt}{v4}\rput(2.8,6.3){$v_4$}
\cnode*(4.5,6){3pt}{v4'}\rput(4.7,6.3){$v_4'$} \ncline{u4}{v4}
\ncline{u4'}{v4'}

\cnode(-3.75,5){3pt}{w1}\rput(-3.45,4.7){$w_1$} \ncline{w1}{v1}
\ncline{w1}{v1'} \cnode(-1.25,5){3pt}{w2}\rput(-.95,4.7){$w_2$}
\ncline{w2}{v2} \ncline{w2}{v2'}
\cnode(1.25,5){3pt}{w3}\rput(1.55,4.7){$w_3$} \ncline{w3}{v3}
\ncline{w3}{v3'} \cnode(3.75,5){3pt}{w4}\rput(4.05,4.7){$w_4$}
\ncline{w4}{v4} \ncline{w4}{v4'}

\cnode*(-3.75,4){3pt}{z1}\rput(-3.75,3.7){$z_1$} \ncline{z1}{u1}
\ncline{z1}{u1'} \ncline{w1}{z1}
\cnode(-1.25,4){3pt}{z2}\rput(-1.25,3.7){$z_2$} \ncline{z2}{u2}
\ncline{z2}{u2'}  \ncline{w2}{z2}
\cnode(1.25,4){3pt}{z3}\rput(1.45,4.3){$z_3$} \ncline{z3}{u3}
\ncline{z3}{u3'}  \ncline{w3}{z3}
\cnode(3.75,4){3pt}{z4}\rput(3.75,3.7){$z_4$} \ncline{z4}{u4}
\ncline{z4}{u4'}  \ncline{w4}{z4}

\cnode(-3.75,6){3pt}{y1}\rput(-3.75,6.3){$y_1$} \ncline{y1}{v1}
\ncline{y1}{v1'}  \nccurve[angleA=-135,angleB=135]{y1}{z1}
\cnode(-1.25,6){3pt}{y2}\rput(-1.25,6.3){$y_2$} \ncline{y2}{v2}
\ncline{y2}{v2'}  \nccurve[angleA=-135,angleB=135]{y2}{z2}
\cnode(1.25,6){3pt}{y3}\rput(1.25,6.3){$y_3$} \ncline{y3}{v3}
\ncline{y3}{v3'}  \nccurve[angleA=-135,angleB=135]{y3}{z3}
\cnode(3.75,6){3pt}{y4}\rput(3.75,6.3){$y_4$} \ncline{y4}{v4}
\ncline{y4}{v4'}  \nccurve[angleA=-135,angleB=135]{y4}{z4}

\cnode*(-3.75,7){3pt}{x1}\rput(-3.75,7.3){$x_1$} \ncline{x1}{v1}
\ncline{x1}{v1'} \cnode(-1.25,7){3pt}{x2}\rput(-1.25,7.3){$x_2$}
\ncline{x2}{v2} \ncline{x2}{v2'}
\cnode(1.25,7){3pt}{x3}\rput(1.25,7.3){$x_3$} \ncline{x3}{v3}
\ncline{x3}{v3'} \cnode(3.75,7){3pt}{x4}\rput(3.75,7.3){$x_4$}
\ncline{x4}{v4} \ncline{x4}{v4'}

\ncline{c1}{u1} \ncline{c1}{u2} \ncline{c1}{u3'} \ncline{c2}{u1'}
\ncline{c2}{u2} \ncline{c2}{u4} \ncline{c3}{u2'} \ncline{c3}{u3}
\ncline{c3}{u4}
\end{pspicture}
\caption{\label{f1}\footnotesize An instance of the Roman bondage
number problem resulting from an instance of 3SAT. Here $k=1$ and
$\gamma_{\rm R}(G)=18$, where the bold vertex $w$ means a Roman
dominating function with $f(w)=2$.}
\end{center}
\end{figure}

To prove that this is indeed a transformation, we only need to
show that $b_{\rm R}(G)=1$ if and only if there is a truth
assignment for $U$ that satisfies all clauses in $\mathscr{C}$.
This aim can be obtained by proving the following four claims.

\begin{description}

\item [Claim 4.1] {\it $\gamma_{\rm R}(G)\geq 4n+2$. Moreover, if
$\gamma_{\rm R}(G)=4n+2$, then for any $\gamma_{\rm R}$-function
$f$ on $G$, $f(H_i)=4$ and at most one of $f(u_i)$ and
$f(\bar{u}_i)$ is 2 for each $i$, $f(c_j)=0$ for each $j$ and
$f(s_2)=2$.}

\begin{proof}
Let $f$ be a $\gamma_{\rm R}$-function of $G$, and let
$H_i'=H_i-u_i-\bar{u}_i$.

If $f(u_i)=2$ and $f(\bar{u}_i)=2$, then $f(H_i)\geq 4$. Assume
either $f(u_i)=2$ or $f(\bar{u}_i)=2$, if $f(x_i)=0$ or
$f(y_i)=0$, then there is at least one vertex $t$ in $\{v_i,
v_i',z_i\}$ such that $f(t)=2$. And hence $f(H_i')\ge 2$. Thus,
$f(H_i)\geq 4$.

If $f(u_i)\neq 2$ and $f(\bar{u}_i)\neq 2$, let $f'$ be a
restriction of $f$ on $H_i'$, then $f'$ is a Roman dominating
function of $H_i'$, and $f'(H_i')\geq \gamma_{\rm R}(H_i')$. Since
the maximum degree of $H_i'$ is $V(H_i')-3$, by Lemma~\ref{propG},
$\gamma_{\rm R}(H_i')>3$ and hence $f'(H_i')\geq 4$ and $f(H_i)\geq
4$. If $f(s_1)=0$ or $f(s_3)=0$, then there is at least one vertex
$t$ in $\{c_1, \cdots,c_m, s_2\}$ such that $f(t)=2$. Then
$f(N_G[V(P)])\geq 2$, and hence $\gamma_{\rm R}(G)\geq 4n+2$.

Suppose that $\gamma_{\rm R}(G)=4n+2$, then $f(H_i)=4$ and since
$f(N_G[x_i])\ge 1$, at most one of $f(u_i)$ and $f(\bar{u}_i)$ is
2 for each $i=1,2,\ldots,n$, while $f(N_G[V(P)])=2$. It follows
that $f(s_2)=2$ since $f(N_G[s_2])\ge 1$. Consequently, $f(c_j)=0$
for each $j=1,2,\ldots,m$.
\end{proof}

\item [Claim 4.2] {\it $\gamma_{\rm R}(G)=4n+2$ if and only if
$\mathscr{C}$ is satisfiable.}

\begin{proof}
Suppose that $\gamma_{\rm R}(G)=4n+2$ and let $f$ be a
$\gamma_{\rm R}$-function of $G$. By Claim 4.1, at most one of
$f(u_i)$ and $f(\bar{u}_i)$ is 2 for each $i=1,2,\ldots,n$. Define
a mapping $t: U\to \{T,F\}$ by
\begin{equation}\label{e4.1}
 t(u_i)=\left\{
 \begin{array}{l}
 T \ \ {\rm if}\ f(u_i)=2\ {\rm or}\ f(u_i)\neq 2\ {\rm and} f(\bar{u}_i)\neq 2, \\
 F \ \ {\rm if}\ f(\bar{u}_i)=2.
\end{array}
 \right.
 \ i=1,2,\ldots,n.
 \end{equation}

We now show that $t$ is a satisfying truth assignment for
$\mathscr{C}$. It is sufficient to show that every clause in
$\mathscr{C}$ is satisfied by $t$. To this end, we arbitrarily
choose a clause $C_j\in\mathscr{C}$ with $1\le j\le m$.

By Claim 4.1, $f(c_j)=f(s_1)=f(s_3)=0$. There exists some $i$ with
$1\le i\le n$ such that $f(u_i)=2$ or $f(\bar{u}_i)=2$ where $c_j$
is adjacent to $u_i$ or $\bar{u}_i$. Suppose that $c_j$ is
adjacent to $u_i$ where $f(u_i)=2$. Since $u_i$ is adjacent to
$c_j$ in $G$, the literal $u_i$ is in the clause $C_j$ by the
construction of $G$. Since $f(u_i)=2$, it follows that $t(u_i)=T$
by (\ref{e4.1}), which implies that the clause $C_j$ is satisfied
by $t$. Suppose that $c_j$ is adjacent to $\bar{u}_i$ where
$f(\bar{u}_i)=2$. Since $\bar{u}_i$ is adjacent to $c_j$ in $G$,
the literal $\bar{u}_i$ is in the clause $C_j$. Since
$f(\bar{u}_i)=2$, it follows that $t(u_i)=F$ by (\ref{e4.1}).
Thus, $t$ assigns $\bar{u}_i$ the truth value $T$, that is, $t$
satisfies the clause $C_j$. By the arbitrariness of $j$ with $1\le
j\le m$, we show that $t$ satisfies all the clauses in
$\mathscr{C}$, that is, $\mathscr{C}$ is satisfiable.

Conversely, suppose that $\mathscr{C}$ is satisfiable, and let $t:
U\to \{T,F\}$ be a satisfying truth assignment for $\mathscr{C}$.
Create a function $f$ on $V(G)$ as follows: if $t(u_i)=T$, then
let $f(u_i)=f(v_i')=2$, and if $t(u_i)=F$, then let
$f(\bar{u}_i)=f(v_i)=2$. Let $f(s_2)=2$. Clearly, $f(G)=4n+2$.
Since $t$ is a satisfying truth assignment for $\mathscr{C}$, for
each $j=1,2,\ldots,m$, at least one of literals in $C_j$ is true
under the assignment $t$. It follows that the corresponding vertex
$c_j$ in $G$ is adjacent to at least one vertex $w$ with $f(w)=2$
since $c_j$ is adjacent to each literal in $C_j$ by the
construction of $G$. Thus $f$ is a Roman dominating function of
$G$, and so $\gamma_{\rm R}(G)\leq f(G)= 4n+2$. By Claim 4.1,
$\gamma_{\rm R}(G)\geq 4n+2$, and so $\gamma_{\rm R}(G)=4n+2$.
\end{proof}

\item [Claim 4.3] {\it $\gamma_{\rm R}(G-e)\leq 4n+3$ for any
$e\in E(G)$.}

\begin{proof}
For any edge $e\in E(G)$, it is sufficient to construct a Roman
dominating function $f$ on $G-e$ with weight $4n+3$. We first
assume $e\in E_G(s_1)$ or $e\in E_G(s_3)$ or $e\in E_G(c_j)$ for
some $j=1,2,\ldots,m$, without loss of generality let $e\in
E_G(s_1)$ or $e=c_ju_i$ or $e=c_j\bar{u}_i$. Let $f(s_3)=2,
f(s_1)=1$ and $f(u_i)=f(v_i')=2$ for each $i=1,2,\ldots,n$. For
the edge $e\notin E_G(u_i)$ and $e\notin E_G(v_i')$, let
$f(s_1)=2, f(s_3)=1$ and $f(u_i)=f(v_i')=2$. For the edge $e\notin
E(\bar{u}_i)$ and $e\notin E(v_i)$, let $f(s_1)=2, f(s_3)=1$ and
$f(\bar{u}_i)=f(v_i)=2$. If $e=u_iv_i$ or $e=\bar{u}_iv_i'$, let
$f(s_1)=2, f(s_3)=1$ and $f(x_i)=f(z_i)=2$. Then $f$ is a Roman
dominating function of $G-e$ with $f(G-e)=4n+3$ and hence
$\gamma_{\rm R}(G-e)\leq 4n+3$.
\end{proof}

\item [Claim 4.4] {\it $\gamma_{\rm R}(G)=4n+2$ if and only if
$b_{\rm R}(G)=1$.}

\begin{proof}
Assume $\gamma_{\rm R}(G)=4n+2$ and consider the edge $e=s_1s_2$.
Suppose $\gamma_{\rm R}(G)=\gamma_{\rm R}(G-e)$. Let $f'$ be a
$\gamma_{\rm R}$-function of $G-e$. It is clear that $f'$ is also
a $\gamma_{\rm R}$-function on $G$. By Claim 4.1 we have
$f'(c_j)=0$ for each $j=1,2,\ldots,m$ and $f'(s_2)=2$. But then
$f'(N_{G-e}[s_1])=0$, a contradiction. Hence, $\gamma_{\rm
R}(G)<\gamma_{\rm R}(G-e)$, and so $b_{\rm R}(G)=1$.

Now, assume $b_{\rm R}(G)=1$. By Claim 4.1, we have  $\gamma_{\rm
R}(G)\geq 4n+2$. Let $e'$ be an edge such that $\gamma_{\rm
R}(G)<\gamma_{\rm R}(G-e')$. By Claim 4.3, we have that
$\gamma_{\rm R}(G-e')\leq 4n+3$. Thus, $4n + 2\leq \gamma_{\rm
R}(G)< \gamma_{\rm R}(G-e')\leq 4n+3$, which yields $\gamma_{\rm
R}(G)=4n+2$.
\end{proof}
\end{description}

By Claim 4.2 and Claim 4.4, we prove that $b_{\rm R}(G)=1$ if and
only if there is a truth assignment for $U$ that satisfies all
clauses in $\mathscr{C}$. Since the construction of the Roman
bondage number instance is straightforward from a
$3$-satisfiability instance, the size of the Roman bondage number
instance is bounded above by a polynomial function of the size of
$3$-satisfiability instance. It follows that this is a polynomial
reduction and the proof is complete.
\end{proof}

\begin{corollary}
Roman domination number problem is NP-complete even for bipartite
graphs.
\end{corollary}

\begin{proof}
It is easy to see that the Roman domination problem is in NP since
a nondeterministic algorithm need only guess a vertex set pair
$(V_1,V_2)$ with $|V_1|+2|V_2|\le k$ and check in polynomial time
whether that for any vertex $u\in V\setminus(V_1\cup V_2)$ whether
there is a vertex in $V_2$ adjacent to $u$ for a given nonempty
graph $G$.

We use the same method as Theorem~\ref{thm4} to prove this
conclusion. We construct the same graph $G$ but does not contain the
path $P$. We set $k=4n$, then use the same methods as Claim 4.1 and
4.2, we have that $\gamma_{\rm R}(G)=4n$ if and only if
$\mathscr{C}$ is satisfiable.
\end{proof}

\section{General bounds}

\begin{lemma}\label{lem6}
Let $G$ be a connected graph of order $n\ge 3$ such that
$\gamma_{\rm R}(G)=\gamma(G)+1$. If there is a set $B$ of edges
with $\gamma_{\rm R}(G-B)=\gamma_{\rm R}(G)$, then
$\Delta(G)=\Delta(G-B)$.
\end{lemma}
\begin{proof}
Since $G$ is connected and $n\geq 3$, $\gamma_{\rm
R}(G)=\gamma(G)+1\leq n-1$. Since $\gamma_{\rm R}(G-B)=\gamma_{\rm
R}(G)\leq n-1$,  $G-B$ is nonempty. It follows from  Propositions
\ref{propC} and \ref{propD} that $\gamma_{\rm R}(G-B)\geq
\gamma(G-B)+1$. Since
$$\gamma_{\rm R}(G-B)=\gamma_{\rm R}(G)=\gamma(G)+1\leq\gamma(G-B)+1,$$
we have $\gamma_{\rm R}(G-B)=\gamma(G-B)+1$, and then
$\gamma(G-B)=\gamma(G)$.

If $G-B$ is connected, then by Proposition \ref{propE},
$$\Delta(G-B)=n-\gamma(G-B)=n-\gamma(G)=\Delta(G).$$

If $G-B$ is disconnected, then let $G_1$ be a nonempty connected
component of $G-B$. By Propositions \ref{propC} and \ref{propD},
$\gamma_{\rm R}(G_1)\geq \gamma(G_1)+1$. Then
$$\begin{array}{ccc}
\gamma(G)+1 &= &\gamma_{\rm R}(G-B)\hfill\\
&= & \gamma_{\rm R}(G_1)+\gamma_{\rm R}(G-G_1)\\
&\geq &\gamma(G_1)+1+\gamma(G-G_1)\hfill\\
&\geq & \gamma(G)+1,\hfill
\end{array}$$
and hence $\gamma_{\rm R}(G_1)=\gamma(G_1)+1$,  $\gamma_{\rm
R}(G-G_1)=\gamma(G-G_1)$ and $\gamma(G)=\gamma(G_1)+\gamma(G-G_1)$.
By Proposition \ref{propD}, $G-G_1$ is empty and hence
$\gamma(G-G_1)=|V(G-G_1)|$. By Proposition \ref{propE},
\begin{align*}
\Delta(G_1)&=|V(G_1)|-\gamma(G_1)\\
&=n-|V(G-G_1)|-\gamma(G_1)\\
&=n-\gamma(G-G_1)-\gamma(G_1)\\
&=n-\gamma(G)=\Delta(G)
\end{align*}
as desirable.
\end{proof}

\begin{theorem}\label{thm7}
Let $G$ be a connected graph of order $n\ge 3$ with $\gamma_{\rm
R}(G)=\gamma(G)+1$. Then $$b_{\rm R}(G)\leq \min
\{b(G),n_{\Delta}\},$$ where $n_{\Delta}$ is the number of
vertices with maximum degree $\Delta$ in $G$.
\end{theorem}

\begin{proof} Since $n\geq 3$ and $G$ is connected, we have
$\Delta(G)\geq 2$ and hence   $\gamma(G)\leq n-2$. Let $B$ be a
$b(G)$- set. By (\ref{eqb}), $\gamma(G-B)=\gamma(G)+1\leq n-1$ and
so $G-B$ is nonempty. It follows from Propositions \ref{propC} and
\ref{propD} that $\gamma_{\rm R}(G-B)\geq
\gamma(G-B)+1>\gamma(G)+1=\gamma_{\rm R}(G)$ and hence $B$ is a
Roman bondage set of $G$. Thus, $b_{\rm R}(G)\le b(G)$.

We now prove that $b_{\rm R}(G)\leq n_{\Delta}$. It follows from
Propositions \ref{propA}, \ref{propE} and the fact $\gamma_{\rm
R}(G) =\gamma(G)+1$ that $\Delta(G)=n-\gamma(G)$. Let
$\{v_1,\ldots,v_{n_{\Delta}}\}$ be the set consists of all vertices
of degree $\Delta$ and let $e_i$ be  an edge adjacent to $v_i$ for
each $1\le i\le n_{\Delta}$.  Suppose
$B'=\{e_1,\ldots,e_{n_{\Delta}}\}$. Clearly,
$\Delta(G-B')<\Delta(G)=n-\gamma(G)$ and $G-B'$ is nonempty. Since
$G-B'$ is nonempty, it follows from Propositions \ref{propC} and
\ref{propD} that $\gamma_{\rm R}(G-B')\geq \gamma(G-B')+1$. We claim
that $\gamma_{\rm R}(G-B')>\gamma_{\rm R}(G)$. Assume to the
contrary that  $\gamma_{\rm R}(G-B')=\gamma_{\rm R}(G)$. We deduce
from Lemma~\ref{lem6} that $\Delta(G-B')=\Delta(G)=n-\gamma(G)$, a
contradiction. Hence $b_{\rm R}(G)\leq |B'|\leq n_{\Delta}$. This
completes the proof.
\end{proof}

\begin{theorem}\label{thm8}
For every Roman graph $G$,
$$b_{\rm R}(G) \geq b(G).$$
The bound is sharp for cycles on $n$ vertices where $n\equiv 0\;({\rm
mod}\;3)$.
\end{theorem}
\begin{proof}
Let $B$ be a $b_R(G)$- set. Then by (\ref{eqq}) we have
$$2\gamma(G-B)\ge \gamma_{\rm R}(G-B)>\gamma_{\rm R}(G)=2\gamma(G).$$
Thus $\gamma(G-B)>\gamma(G)$ and hence $b_{\rm R}(G) \geq b(G)$.

By Proposition~\ref{propH}, we have $b_{\rm R}(C_n) \geq b(C_n)=2$
when $n\equiv 0\;({\rm mod}\;3)$.
\end{proof}

The strict inequality in Theorem~\ref{thm8} can hold, for example,
$b(C_{3k+2})=2<3=b_{\rm R}(C_{3k+2})$ by Proposition~\ref{propH}.

A graph $G$ is called to be {\it vertex domination-critical} (
{\it vc-graph} for short) if $\gamma(G-x) < \gamma(G)$ for any
vertex $x$ in $G$. We call  a graph $G$ to be {\it vertex Roman
domination-critical} ({\it vrc-graph} for short) if $\gamma_{\rm
R}(G-x)<\gamma_{\rm R}(G)$ for every vertex $x$ in $G$.

The {\it vertex covering number} $\beta(G)$ of $G$ is the minimum
number of vertices that are incident with all edges in $G$. If $G$
has no isolated vertices, then $\gamma_{\rm R}(G)\leq 2\gamma(G)
\leq 2\beta(G)$. If $\gamma_{\rm R}(G)=2\beta(G)$, then
$\gamma_{\rm R}(G)=2\gamma(G)$ and hence $G$ is a Roman graph.
In~\cite{v94}, Volkmann gave a lot of graphs with
$\gamma(G)=\beta(G)$.

\begin{theorem}\label{thm9}
Let $G$ be a graph with $\gamma_{\rm R}(G)=2\beta(G)$. Then\\
(1) $b_{\rm R}(G)\geq \delta(G)$;\\
(2) $b_{\rm R}(G)\geq \delta(G)+1$ if $G$ is a vrc-graph.
\end{theorem}

\begin{proof}
Let $G$ be a graph such that $\gamma_{\rm R}(G)=2\beta(G)$.

(1)  If $\delta(G)=1$, then the result is immediate. Assume
$\delta(G)\geq 2$. Let $B\subseteq E(G)$ and $|B|\leq
\delta(G)-1$. Then $\delta(G-B)\geq 1$ and so $\gamma_{\rm
R}(G)\leq \gamma_{\rm R}(G-B) \leq 2\beta(G-B) \leq 2
\beta(G)=\gamma_{\rm R}(G)$. Thus, $B$ is not a Roman bondage set
of $G$, and hence $b_{\rm R}(G)\geq \delta(G)$.

(2)Let $B$ be a Roman bondage set of $G$. An argument similar to
that described in the proof of (1), shows that $B$ must contain
all edges incident with some vertex of $G$, say $x$. Hence, $G-B$
has an isolated vertex. On the other hand, since $G$ is a
vrc-graph, $\gamma_{\rm R}(G-x) <\gamma_{\rm R}(G)$ which implies
that the removal of all edges incident to $x$ can not increase the
Roman domination number. Hence, $b_{\rm R}(G)\geq \delta(G)+1$.
\end{proof}
The {\em cartesian product} $G=G_1\times G_2$ of two disjoint
graphs $G_1$ and $G_2$ has $V(G)=V(G_1)\times V(G_2)$, and two
vertices $(u_1,u_2)$ and $(v_1,v_2)$ of $G$ are adjacent if and
only if either $u_1=v_1$ and $u_2v_2\in E(G_2)$ or $u_2=v_2$ and
$u_1v_1\in E(G_1)$. The cartesian product of two paths
$P_r=x_1x_2\ldots x_r$ and $P_t=y_1y_2\ldots y_t$ is called a {\em
grid}. Let $G_{r,s}=P_r\times P_t$ is a grid, and let
$V(G_{r,s})=\{u_{i,j}=(x_i,y_j)|1\le i\le r\,\,{\rm and}\,\,1\le
j\le t\}$ be the vertex set of $G$. Next we determine Roman
bondage number of grids.

\begin{theorem}\label{thm10}
For $n\ge 2$, $b_{\rm R}(G_{2,n})=2$.
\end{theorem}

\begin{proof}
By Proposition \ref{propB}, we have $\gamma_{\rm R}(G_{2,n})=n+1$.
Since
$$\gamma_{\rm R}(G_{2,n}-u_{1,1}u_{1,2}-u_{2,1}u_{2,2})=2+\gamma_{\rm R}(G_{2,n-1})=n+2,$$
we deduce that $b_{\rm R}(G_{2,n})\leq 2$. Now we show that
$\gamma_{\rm R}(G_{2,n}-e)= \gamma_{\rm R}(G_{2,n})$ for any edge
$e\in E(G_{2,n})$. Consider two cases.

\smallskip
\noindent {\bf Case 1}\quad $n$ is odd.\\
For $i=1,2,3,4$, define $f_i:V(G_{2,n})\rightarrow \{0,1,2\}$ as
follows:
$$f_1(u_{i,j})=\left\{\begin{array}{ccc}
  2 & {\rm if} & i=1 \;{\rm and}\; j\equiv 1\;({\rm mod}\;4)\;\;{\rm or}\;\; i=2 \;{\rm and}\; j\equiv 3\;({\rm mod}\;4)\\
  0 & {\rm if} & {\rm otherwise},\hfill
\end{array}\right.$$

$$f_2(u_{i,j})=\left\{\begin{array}{ccc}
  2 & {\rm if} & i=1 \;{\rm and}\; j\equiv 3\;({\rm mod}\;4)\;\;{\rm or}\;\; i=2 \;{\rm and}\; j\equiv 1\;({\rm mod}\;4)\\
  0 & {\rm if} & {\rm otherwise},\hfill
\end{array}\right.$$ and if $n\equiv 1\;({\rm mod}\;4) $, then

$$f_3(u_{i,j})=\left\{\begin{array}{ccc}
  2 & {\rm if} & i=1 \;{\rm and}\; j\equiv 0\;({\rm mod}\;4)\;\;{\rm or}\;\; i=2 \;{\rm and}\; j\equiv 2\;({\rm mod}\;4)\\
  1 & {\rm if} & i=j=1 \;\;{\rm or}\;\;i=2 \;{\rm and}\; j=n \hfill\\
  0 & {\rm if} & {\rm otherwise}.\hfill
\end{array}\right.$$

and if $n\equiv 3\;({\rm mod}\;4) $, then

$$f_4(u_{i,j})=\left\{\begin{array}{ccc}
  2 & {\rm if} & i=1 \;{\rm and}\; j\equiv 2\;({\rm mod}\;4)\;\;{\rm or}\;\; i=2 \;{\rm and}\; j\equiv 0\;({\rm mod}\;4)\\
  1 & {\rm if} & i=2 \;{\rm and}\; j=1\;\;{\rm or}\;\;i=2 \;{\rm and}\; j=n\hfill\\
  0 & {\rm if} & {\rm otherwise}.\hfill
\end{array}\right.$$

Obviously, $f_i$ is a $\gamma_R(G_{2,n})$-function for each
$i=1,2,3$ when $n\equiv 1\;({\rm mod}\;4)$ and $f_i$ is a
$\gamma_R(G_{2,n})$-function for each $i=1,2,4$ when $n\equiv
3\;({\rm mod}\;4)$. Let  $e\in E(G)$ be an arbitrary edge of $G$.
Then clearly , $f_1$ or $f_2$ or $f_3$ is a Roman dominating
function of $G-e$ if $n\equiv 1\;({\rm mod}\;4) $ and $f_1$ or
$f_2$ or $f_3$ is a Roman dominating function of $G-e$ if $n\equiv
3\;({\rm mod}\;4) $. Hence $b_R(G_{2,n})\ge 2$.

\smallskip
\noindent {\bf Case 2}\quad $n$ is even.\\
For $i=1,2,3,4$, define $f_i:V(G_{2,n})\rightarrow \{0,1,2\}$ as
follows:
$$f_1(u_{i,j})=\left\{\begin{array}{ccc}
  2 & {\rm if} & i=1 \;{\rm and}\; j\equiv 0\;({\rm mod}\;4)\;\;{\rm or}\;\; i=2 \;{\rm and}\; j\equiv 2\;({\rm mod}\;4)\\
  1 & {\rm if} & i=j=1\hfill\\
  0 & {\rm if} & {\rm otherwise},\hfill
\end{array}\right.$$

$$f_2(u_{i,j})=\left\{\begin{array}{ccc}
  2 & {\rm if} & i=1 \;{\rm and}\; j\equiv 2\;({\rm mod}\;4)\;\;{\rm or}\;\; i=2 \;{\rm and}\; j\equiv 0\;({\rm mod}\;4)\\
  1 & {\rm if} & i=2 \;{\rm and}\; j=1\hfill\\
  0 & {\rm if} & {\rm otherwise}.\hfill
\end{array}\right.$$ and if $n\equiv 0\;({\rm mod}\;4)$, then

$$f_3(u_{i,j})=\left\{\begin{array}{ccc}
  2 & {\rm if} & i=1 \;{\rm and}\; j\equiv 1\;({\rm mod}\;4)\;\;{\rm or}\;\; i=2 \;{\rm and}\; j\equiv 3\;({\rm mod}\;4)\\
  1 & {\rm if} & i=1 \;{\rm and}\; j=n\hfill \\
  0 & {\rm if} & {\rm otherwise},\hfill
\end{array}\right.$$ and if $n\equiv 2\;({\rm mod}\;4)$, then

$$f_4(u_{i,j})=\left\{\begin{array}{ccc}
  2 & {\rm if} & i=1 \;{\rm and}\; j\equiv 1\;({\rm mod}\;4)\;\;{\rm or}\;\; i=2 \;{\rm and}\; j\equiv 3\;({\rm mod}\;4)\\
  1 & {\rm if} & i=2 \;{\rm and}\; j=n\hfill \\
  0 & {\rm if} & {\rm otherwise},\hfill
\end{array}\right.$$

Obviously, $f_i$ is a $\gamma_R(G_{2,n})$-function for each
$i=1,2,3$ when $n\equiv 0\;({\rm mod}\;4)$ and $f_i$ is a
$\gamma_R(G_{2,n})$-function for each $i=1,2,4$ when $n\equiv
2\;({\rm mod}\;4)$. Let  $e\in E(G)$ be an arbitrary edge of $G$.
Then clearly , $f_1$ or $f_2$ or $f_3$ is a Roman dominating
function of $G-e$ if $n\equiv 0\;({\rm mod}\;4) $ and $f_1$ or
$f_2$ or $f_4$ is a Roman dominating function of $G-e$ if $n\equiv
2\;({\rm mod}\;4) $. Hence $b_R(G_{2,n})\ge 2$. This completes the
proof.
\end{proof}

\section{Roman bondage number of graphs with small Roman domination number}
Dehgardi, Sheikholeslami and Volkmann \cite{DSV} posed the
following problem: If $G$ is a connected graph of order $n\ge 4$
with Roman domination number $\gamma_R(G)\ge 3$, then
\begin{equation}\label{eqqq}b_{R}(G)\le (\gamma_R(G)-2)\Delta(G).\end{equation}
Theorem \ref{propG} shows that the inequality (\ref{eqqq}) holds if
$\gamma_R(G)\ge 5$. Thus the bound in (\ref{eqqq}) is of interest
only when $\gamma_R(G)$ is 3 or 4. In this section we prove
(\ref{eqqq}) for all graphs $G$ of order $n\ge 4$ with
$\gamma_R(G)=3,4$, improving Proposition \ref{propG}.

\begin{theorem}\label{thm11}
If $G$ is a connected graph of order $n\ge 4$ with
$\gamma_R(G)=3$, then
$$b_{R}(G)\le\Delta(G)=n-2.$$
\end{theorem}
\begin{proof}
Let $\gamma_R(G)=3$. Then $\Delta(G)=n-2$ by Proposition
\ref{propG}. Let $M$ be maximum matching of $G$ and let $U$ be the
set consisting of unsaturated vertices. Since $G$ is connected and
$\gamma_R(G)=3$, we deduce that $|M|\ge 2$.

If $U=\emptyset$, then $G-M$ has no vertex of degree $n-2$ and it
follows from Proposition \ref{propG} that $\gamma_R(G-M)\ge 4$. Thus
\begin{equation}\label{eq1}b_R(G)\le |M|\le \frac{n}{2}\le n-2=\Delta(G).\end{equation}

Assume now that $U\neq\emptyset$. Clearly $U$ is an independent set.
Since $G$ is connected and $M$ is maximum, there exist a set $J$ of
$|U|$ edges such that each vertex of $U$ is incident with exactly
one edge of $J$. Then $|J|=|U|=n-2|M|$. Now let $F=J\cup M$.
Obviously, $G-F$ has no vertex of degree $n-2$, and it follows from
Proposition \ref{propG} that $\gamma_R(G_F)\ge 4$. This implies that
\begin{equation}\label{eq2}b_R(G)\le |M|+|U|=n-|M|\le n-2=\Delta(G).\end{equation}
This completes the proof.
\end{proof}

Next we characterize all graphs that achieve the bound in Theorem
\ref{thm11}.
\begin{theorem}\label{thm12}
If equality holds in Theorem \ref{thm11}, then $G$ is regular.
\end{theorem}
\begin{proof}
Let $\gamma_R(G)=3$ and $b_R(G)=\Delta(G)=n-2$. If $G$ has a
perfect matching $M$, then it follows from (\ref{eq1}) that
$\frac{n}{2}=n-2$ and hence $n=4$. This implies that
$b_R(G)=|M|=2=\Delta(G)$. Since $b_R(P_4)=1$, we have $G=C_4$ as
desired.

Let $G$ does not have a perfect matching and let $M$ be a maximum
matching of $G$. It follows from (\ref{eq2}) that $|M|=2$. Let $X$
be the independent set of $M$-unsaturated vertices. We consider
two cases.

\smallskip
\noindent {\bf Case 1.} $|X|=1$.\\ Then  $n=5$. Let
$V(G)=\{v_1,\ldots,v_5\}$. Since $\gamma_R(G)=3$, $\Delta(G)=n-2=3$
by Proposition \ref{propG}. Since $n$ is odd, $G$ has a vertex of
even degree 2. Let $\deg(v_1)=2$ and let $v_1v_2, v_1v_3\in E(G)$.
Since $b_R(G)=3>\deg(v_1)$, we have
$\gamma_R(G-v_1)=\gamma_R(G)-1=2$. By Observation \ref{ob1},
$\Delta(G-v_1)=3$. Since $\gamma_R(G)=3$, we may assume without loss
of generality that $\deg(v_4)=3$ and $\{v_4v_2,
v_4v_3,v_4v_5\}\subseteq E(G)$. Let $F=\{v_1v_2, v_3v_4\}$. Since
$b_R(G)=3>|F|$, we have $\gamma_R(G-F)=3$. It follows from
Proposition \ref{propG} and the fact $\gamma_R(G-F)=3$ that
$\deg_{G-F}(v_5)=3$. This implies that
$\{v_5v_2,v_5v_3,v_5v_4\}\subseteq E(G)$. Thus
$E(G)=\{v_1v_2,v_1v_3,v_2v_4,v_2v_5,v_3v_4,v_3v_5,v_4v_5\}$. Now we
have $G-\{v_2v_4,v_3v_5\}\simeq C_5$ and hence
$\gamma_R(G-\{v_2v_4,v_3v_5\})=4$. This implies that $b_R(G)\le 2$ a
contradiction.

\smallskip
\noindent {\bf Case 2.}\quad $|X|\ge 2$. \\ Then $n\ge 6$. Let
$M=\{u_1v_1, u_2v_2\}$ be a maximum matching of $G$. If $y$ and $z$
are vertices of $X$ and $yu_i\in E(G)$, then since the matching $M$
is maximum, $zv_i\notin E(G)$. Therefore, we may assume without loss
of generality that $N_G(X)\subseteq \{u_1,u_2\}$. So
$\deg(y)+\deg(z)\le 4$ for every pair of distinct vertices $y$ and
$z$ in $X$. Let $y,z\in X$ and $F$ be the set of edges incident with
$y$ or $z$. Then $y,z$ are isolated vertices in $G-F$ and hence
$\gamma_R(G-F)\ge 4$. If $|F|\le 3$, then $n-2=b_R(G)\le 3$ which
leads to a contradiction. Therefore, $|F|=4$. It follows that
$n-2=b_R(G)\le 4$ and hence $n= 6$. Let
$V(G)=\{u_1,u_2,v_1,v_2,y,z\}$. Then $\deg(y)=\deg(z)=2$ and
$\deg(u_1),\deg(u_2)\ge 3$. If $v_1v_2\in E(G)$, then
$\{yu_1,xu_2,v_1v_2\}$ is a matching of $G$ which is a
contradiction. Thus $\deg(v_1),\deg(v_2)\le 2$. Since
$\gamma_R(G)=3$, $\Delta(G)=n-2=4$ by Proposition \ref{propG}. We
distinguish two subcases.

\smallskip
{\bf Subcase 2.1}\quad $\delta(G)=1$. \\ Assume  without loss of
generality that $\deg(v_1)=1$. Let $F$ be the set of edges
incident with $y$ or $v_1$. Then $|F|=3$ and $y,v_1$ are isolated
vertices in $G-F$ and hence $\gamma_R(G-F)\ge 4$. Thus
$n-2=b_R(G)\le 3$, a contradiction.

\smallskip
{\bf Subcase 2.2}\quad $\delta(G)=2$.\\ Then we must have
$\deg(v_1)=\deg(v_2)=2$ and  $v_1u_2,v_2u_1\in E(G)$. Let
$F=\{yu_1,zu_2\}$. Clearly $\Delta(G-F)=3=n-3$ and it follows from
Proposition \ref{propG} that $\gamma_R(G-F)\ge 4$. Hence $b_R(G)\le
2$, which is a contradiction.

This completes the proof.
\end{proof}

\begin{prelem}\label{propJ}
The complete graph $K_{2r}$ is 1-factorable.
\end{prelem}

According to Theorem \ref{thm11}, Theorem \ref{thm12}, Proposition
\ref{propG} and Proposition \ref{propJ}, we prove the next result.

\begin{theorem}\label{thm13}
Let $G$ be a connected graph of order $n\ge 4$ with
$\gamma_R(G)=3$. Then $b_{R}(G)=\Delta(G)= n-2$ if and only if
$G\simeq C_4$.
\end{theorem}
\begin{proof}
Let $G$ be a connected graph of order $n\ge 4$ with $\gamma_R(G)=3$.
It follows from Theorem \ref{thm11} that $b_{R}(G)\le n-2$.

If $G\simeq C_4$, then obviously $b_R(G)=2=n-2$.

Conversely, assume that $b_R(G)=n-2$.  It follows from Proposition
\ref{propG} and Theorem \ref{thm12} that $G$ is $(n-2)$-regular.
This implies that $n$ is even and hence $G=K_n-M$ where $M$ is a
perfect matching in $K_n$. By Proposition \ref{propJ}, $G$ is
1-factorable. Let $M_1$ be a perfect matching in $G$. Now $G-M_1$ is
an $(n-3)$-regular and it follows from Proposition \ref{propG} that
$\gamma_R(G-M_1)\ge 4$. Thus $n-2=b_R(G)\le \frac{n}{2}$ which
implies that $n=4$ and hence $G=C_4$.
\end{proof}

\begin{theorem}\label{thm14}
If $G$ is a connected graph of order $n\ge 4$ with
$\gamma_R(G)=4$, then
$$b_{R}(G)\le \Delta(G)+\delta(G)-1.$$
\end{theorem}

\begin{proof}
Obviously $\Delta(G)\ge 2$. Let $u$ be a vertex of minimum degree
$\delta(G)$.  If $b_R(G)\le \deg(u)$, then we are done.  Suppose
$b_{R}(G)> \deg(u)$. Then $\gamma_R(G-u)=\gamma_R(G)-1=3$. By
Theorem \ref{thm11}, $b_R(G-u)\le \Delta(G-u)$. If $b_R(G-u)=
\Delta(G-u)$, then $G-u=C_4$ by Theorem \ref{thm13} and since $G$ is
connected, we deduce that $\gamma_R(G)=3$, a contradiction. Thus
$b_R(G-u)\le \Delta(G-u)-1$. It follows from Observation \ref{ob2}
that \begin{equation}\label{eq3}b_R(G)\le b_R(G-u)+\deg(u)\le
\Delta(G-u)-1+\deg(u) \le \Delta(G)+\delta(G)-1,\end{equation} as
desired. This completes the proof.
\end{proof}

Dehgardi et al. \cite{DSV} proved that  for any connected graph
$G$ of order $n\ge 3$, $b_{\rm R}(G)\le n-1$  and posed the
following problems.

\smallskip
\noindent {\bf Problem 1.}\label{p1} {\em Prove or disprove: For any
connected graph $G$ of order $n\ge 3$, $b_{\rm R}(G)=n-1$ if and
only if $G\cong K_3$.}

\smallskip
\noindent {\bf Problem 2.} {\em Prove or disprove: If $G$ is a
connected graph of order $n\ge 3$, then
$$b_{\rm R}(G)\le n-\gamma_{\rm R}(G)+1.$$}

Since $\gamma_R(K_{3,3,\ldots,3})=4$, Proposition \ref{propF} shows
that Problems 1 and 2 are false. Recently Akbari and Qajar \cite{AQ}
proved that:

\begin{prelem}
If $G$ is a connected graph of order $n\ge 3$, then
$$b_{R}(G)\le n-\gamma_R(G)+5.$$
\end{prelem}

We conclude this paper with the following revised problems.

\smallskip
\noindent {\bf Problem 3.} {\em Characterize all connected graphs
$G$ of order $n\ge 3$ for which $b_R(G)=n-1$.}\vspace{2mm}

\smallskip
\noindent {\bf Problem 4.}  {\em  Prove or disprove: If $G$ is a
connected graph of order $n\ge 3$, then
$$b_{R}(G)\le n-\gamma_R(G)+3.$$}

{}

\end{document}